\documentclass[a4paper]{amsart}
\usepackage{amssymb}
\usepackage{verbatim}
\usepackage{amscd}
\def\card{{{\operatorname{card}}}}

\numberwithin{equation}{section}
\theoremstyle{plain}

\newtheorem*{(DQ1)}{(DQ1)}
\newtheorem*{lemma a}{Lemma (a)}
\newtheorem*{lemma b}{Lemma (b)}
\newtheorem*{lemma c}{Lemma (c)}
\newtheorem*{embedding lemma}{Lemma}
\newtheorem*{embedding}{Theorem}
\newtheorem*{lemma1}{Lemma 1}
\newtheorem*{lemma2}{Lemma 2}

\newtheorem*{proposition0}{Proposition M0}
\newtheorem*{proposition1}{Proposition M1}

\theoremstyle{definition}

\theoremstyle{remark}

\newtheorem*{Proposition 1}{Proposition 1}
\begin{document}
\title{Embedding subshifts of finite type into the Fibonacci-Dyck shift}
\author{Toshihiro Hamachi}
\author{Wolfgang Krieger}
\begin{abstract}
A necessary and sufficient condition is given for the existence of an embedding of an irreducible subshift of finite type into the Fibonacci-Dyck shift 
 \end{abstract}
 
\maketitle

%section 1
\section{Introduction}

Let $\Sigma$ be a finite alphabet, and let $S$ be the shift 
on the shift space $\Sigma^{\Bbb Z}$,

$$
S((x_{i})_{i \in \Bbb Z}) =  (x_{i+1})_{i \in \Bbb Z}, 
\qquad 
(x_{i})_{i
\in \Bbb Z}  \in \Sigma^{\Bbb Z}.
$$
An 
$S$-invariant closed subset $X$ of $\Sigma^{\Bbb Z}$ is called a subshift. For an introduction to the theory 
of subshifts see \cite {Ki} or  \cite{LM}. A word is called admissible for the subshift $X \subset  \Sigma^{\Bbb Z}$ if it appears in a point of $X$. A subshift is uniquely determined by its language of admisible word. 

Among the first examples  of subshifts are the subshifts of finite type. A subshift of finite type is  constructed from a finite set $\mathcal F$ of words in the alphabet $\Sigma$ as the subshift that contains the points in $\Sigma^{\Bbb Z}$, in which no word in $\mathcal F$ appears. Other prototypical examples of subshifts are the Dyck shifts. To recall the construction of the Dyck shifts \cite {Kr1}, let $N> 1$, and let 
$$
\alpha_-(n), \alpha_+(n),\quad  0 \leq n < N,
$$
be the generators of the Dyck inverse monoid \cite {NP} $\mathcal D_N$ with the relations
$$
\alpha_-(n) \alpha_+(m) =
\begin{cases}
\bold 1, &\text{if  $n = m$}, \\
0, &\text {if $n \neq m$}.
\end{cases}
$$
The Dyck shifts  are defined as the subshifts
$$
D_N \subset( \{ \alpha_-(n): 0 \leq n < N \} \cup\{ \alpha_+(n):0 \leq n < N \})^\Bbb Z
$$
with the admissible words $(\sigma_i)_{1 \leq i \leq I  } , I \in \Bbb N,$ of $D_N, N > 1,$ given by the condition 
$$
\prod_{1 \leq i \leq I } \sigma_i \neq 0.
$$
In \cite{HI} a necessary and sufficient condition was given for the existence of an embedding of an irreducible subshift of finite type into a Dyck shift. In \cite {HIK} this result was extended to a wider class of target shifts that contains the $ \mathcal D_N$-presentations. 
With  the semigroup $\mathcal D_N^-$ ($\mathcal D_N^+$) that is freely  generated by $ \{ \alpha_-(n): 0 \leq n < N \} $ $ ( \{ \alpha_+(n): 0 \leq n < N \} )$, 
$\mathcal D_N$-presentations can be described as arising from a finite irreducible directed labelled graph with vertex set $\mathcal V$ and edge set $\Sigma$ and a label map 
$$
\lambda: \Sigma \rightarrow \mathcal D_N^- \cup \{\bold 1\} \cup\mathcal D_N^+,
$$
that extends
to directed paths 
$
b = (b_i)_{1\leq i \leq I}, I > 1,
$
in the directed graph $(\mathcal V,\Sigma)$
by 
$
\lambda (b) = \prod_{1\leq i \leq I} \lambda(b_i).
$
It is required that there exists for  $U, W \in \mathcal V$, and for $\beta \in \mathcal D_N,$  a path $b$ from $U$ to $W$ such that $\lambda(b) =Ê\beta$. 
The $\mathcal D_N$-presentation $X(\mathcal V, \Sigma, \lambda)$ is the subshift with alphabet $\Sigma$ and with the set of admissible words given by the set of directed finite paths $b$ in the graph $(\mathcal V, \Sigma, \lambda)$ such that $\lambda (b) \neq 0$. 

For Dyck shifts the notion of a multiplier was introduced in \cite {HI}. A multiplier of the Dyck shift $D_N, Ý N\in \Bbb N,$ or more generally of a ${\mathcal D_N}$-presentation \cite {HIK}, is an equivalence class of primitive words in the symbols $\alpha(n), 0 \leq n \leq N$. Here a word is called primitive if it is not the power of another word, and two primitive words are equivalent, if one is a cyclic permutation of the other. The multipliers are the primitive necklaces of combinatorics \cite [Section 4]{BP}. As a notation for a multiplier we use one of its representatives.

We denote the period of a periodic point $p$ of a $\mathcal D_N$-presentation $X(\mathcal V, \Sigma, \lambda)$ by  $\Pi(p)$. A periodic point $p =(p_i)_{i \in \Bbb Z}$ of a $\mathcal D_N$-presentation $X(\mathcal V, \Sigma, \lambda)$, and its orbit, are said to be 
neutral if there exists an $i \in \Bbb Z$ such that $\lambda( (p_j)_{i\leq j < i+\Pi(p)})=
 \bold 1 $, and they are said to have negative (positive) multplier, if there exists an $i \in \Bbb Z$ such that $\lambda( (p_j)_{i\leq j < i+\Pi(p)})\in$ $\mathcal D_N^-$$ ( \mathcal D_N^+)$.
More precisely, given  a multiplier $\mu$,  a periodic point $p =(p_i)_{i \in \Bbb Z}$ of a
$\mathcal D_N$-presentation $X(\mathcal V, \Sigma, \lambda)$, and its orbit, are said to have  (negative)  multiplier $\mu_-$, if  there exists an $i \in \Bbb Z$ and a representatve 
$(\alpha(n_j))_{1 \leq j \leq J}$ 
of $\mu$
such that $\lambda( (p_j)_{i\leq j < i+\Pi(p)})$ is equal to $\prod_{1 \leq j \leq J}\alpha_-(n_j)
$,
and are said to have  (positive)  multiplier $\mu_+$, if  there exists an $i \in \Bbb Z$ and a representative 
$(\alpha(n_j))_{1 \leq j \leq J}$ 
of $\mu$
such that $\lambda( (p_j)_{i\leq j < i+\Pi(p)})$ is equal to 
$\prod_{J \geq  j \geq 1}\alpha_+(n_j)$.
 Denote the set of  periodic orbits   of length  $n$ of the  $\mathcal D_N$-presentation 
 $X(\mathcal V, \Sigma, \lambda)$ with negative (positive) multiplier  by 
 $\mathcal O_n^-(X(\mathcal V, \Sigma, \lambda)$    ($\mathcal O_n^+(X(\mathcal V, \Sigma, \lambda)$), 
  and  denote the  set of its periodic orbits  of length $n$ with  multiplier $\mu_-$($\mu_+$)  by  
  $\mathcal O_n( \mu_-)(X(\mathcal V, \Sigma, \lambda)$    ($\mathcal O_n( \mu_+)  (X(\mathcal V, \Sigma, \lambda)$).
The notion of an exceptional multiplier was introduced in \cite{HI}.
 A multiplier $\mu$ of a $\mathcal D_N$-presentation $X(\mathcal V, \Sigma, \lambda)$ is said to be exceptional at period $n\in \Bbb N$ if
\begin{multline*}
\card  ( \mathcal O_n( \mu_-)(X(\mathcal V, \Sigma, \lambda)) \cup \mathcal O_n( \mu_+)(X(\mathcal V, \Sigma, \lambda)) ) > \\
\max \{ \card (\mathcal O_n^-(X(\mathcal V, \Sigma, \lambda))), \card (\mathcal O_n^+(X(\mathcal V, \Sigma, \lambda))) \}.
\end{multline*}

The Fibonacci-Dyck shift  is the Markov-Dyck shift \cite {M} of the Fibonacci graph. We introduce the  Fibonacci graph as the directed graph with vertex set $\{ 0, 1 \}$ and edge set 
$\{\beta (0), \beta, \beta(1)\}$: $\beta (0)$ is a loop at vertex $0$, the edge 
$\beta (1)$ goes from vertex $1$ to vertex $0$, and the edge $\beta$ from vertex $0$ to vertex $1$. 
Let $(  \{ 0, 1 \},\{\beta^- (0), \beta^-, \beta^-(1)\}  ) $ be a copy of the Fibonacci graph, and let 
$(  \{ 0, 1 \},\{\beta^+ (0), \beta^+, \beta^+(1)\} ) $  be its reversal. The Fibonacci-Dyck shift  is given by a $\mathcal D_2$-presentation, with
$$
 \mathcal V = \{0, 1  \},  \quad
\Sigma = \{\beta^- (0),\beta^- ,\beta^- (1), \beta^+ (0),\beta^+ ,\beta^+ (1)  \},
$$
and 
$$
\lambda: \Sigma \rightarrow \{\alpha_-(0) ,\alpha_-(1), \alpha_+(0) ,\alpha_+(1)  \},
$$
given by
\begin{align*}
\lambda ( \beta^- (0) ) =  \alpha_- (0) ,   \lambda (\beta^+ (0)  ) =\alpha_+ (0),
\end{align*}
$$
\lambda ( \beta^- (1) ) =  \alpha_- (1) ,   \lambda (\beta^+ (1)  ) =\alpha_+ (1),
$$
$$
\lambda ( \beta_-) =  \lambda (  \beta_+ ) = \bold 1.
$$
For this $\mathcal D_2$-presentation we chose a label map $\lambda$ that assigns the label $\bold 1$ to the edges 
$\beta^-$  and $\beta^+$. These edges $\beta^-$  and $\beta^+$ arise from the splitting of the edge $\beta$ in the Fibonacci graph. This edge $\beta$  is contracted to a vertex in the procedure that turns the Fibonacci graph into the 1-vertex graph with two loops,
whose graph inverse semigroup is $\mathcal D_2$. 
(For this part of the theory see \cite {Kr2, HK}. In \cite [Section 2] {HK} it is shown that the Fibonacci-Dyck shift has Property $(A)$, and in \cite [Section 3] {HK}  it is shown that its associated semigrop is 
$\mathcal D_2$.
The procedure is also described in \cite [Section 3]{HK}).
In this paper we consider exclusively the Fibonacci-Dyck shift $F$. 

The label map $\lambda$ can be written in the form of a matrix with entries in the semigroup ring of 
$\mathcal D_2$:
$$
\begin{pmatrix}
a^-(0) +    a^+(0)               & \bold 1 +a^+(1)
\\
 \bold 1 +a^-(1) & 0
\end{pmatrix}.
$$
Taking the adjoint and applying the involution of the semigroup ring  of $\mathcal D_2$ to its entries, does not change the matrix. This symmetry property of the matrix makes visible the time reversal 
$\rho$ of $F$, that is given by  setting
$$
\chi (\beta^-  ) = \beta^+ ,
\chi (\beta^+  ) = \beta^- ,
$$
$$
\chi (\beta^- (0) ) = \beta^+(0) ,
\chi (\beta^+(0)  ) = \beta^-(0) ,
$$
$$
\chi (\beta^-(1)  ) = \beta^+(1) ,
\chi (\beta^+(1)  ) = \beta^-(1) ,
$$
and
$$
\rho ((x_i)_{i\in \Bbb Z}) = (\chi(x_{-i}))_{i\in \Bbb Z}, \qquad x \in F.
$$ 

We denote the set of multipliers 
 of $F$ by 
$\mathcal M$.
The time reversal $\rho$ maps the set $\mathcal O_n(\mu_- )$ bijectively onto the set $\mathcal O_n(\mu_+ )$ and we can note 
a lemma:
\begin{lemma a}
  \begin{align*}
  \card (\mathcal O_n(\mu_- ) ) = \card ( \mathcal O_n(\mu_+ )), \qquad  n \in \Bbb N,  \mu \in \mathcal M.
  \end{align*} 
\end{lemma a}

We also note orbit counts of the Fibonacci-Dyck shift for small periods as a lemma:
  \begin{lemma b} 
   \begin{align*}
   & 
 \card ( \mathcal O _3 (\alpha^-(0)  )    ) = 2,
 \\
  &  \card ( \mathcal O _5 (\alpha^-(0) )    )  = 9,
   \\
 &\card ( \mathcal O _5 (\alpha^-(0)\alpha^-(1)  )    ) = 4.  
  \\
 \
  \\
   &\card ( \mathcal O _4 (\alpha^-(0)  )    )=2,
 \\
&  \card ( \mathcal O _4 (\alpha^-(1) )    ) = 2 ,
 \\
& \card ( \mathcal O _6 (\alpha^-(0)   )    )= 11,
\\
&\card ( \mathcal O _6 (\alpha^-(1) )    )=11.
\end{align*}

 \end{lemma b}
 
 We note a consequence of Lemma (a) also as a lemma: 
 \begin{lemma c}
A multiplier $\mu \in \mathcal M$  is  exceptional at period $n\in \Bbb N$ if and only if
$$
   \card \ ( \mathcal O_n(\mu_-) )>  \card (\bigcup_{\widetilde{\mu} \in \mathcal M\setminus
   \{ \mu \}}  \mathcal O_n(\widetilde{\mu}_-)  ).
$$
 \end{lemma c}
  
The Fibonacci-Dyck shift has the exceptional multiplier $\alpha(0)$, that is exceptional at period one (and in view of Lemma (b) also at periods three, and five),  and the exceptional multiplier $\alpha(1)$, that is exceptional at period two.
After introducing notation and terminology in section 2, we show in section 3  
that the multiplier $\alpha(0)\alpha(1)$  is not exceptional.  In section 4, we prove that the remaining multipliers are not exceptional. Based on these results, and on the results of \cite {HIK}, we give in section 5 a necessary and sufficient condition for the existence of an embedding of an irreducible subshift of finite type into the Fibonacci-Dyck shift. The multiplier, especially the exceptional multiplier enter here in an essential way.
Moreover, we show  in section 6, that the multiplier $\alpha (0)$ is exceptional  only at periods one, three, and five,  and in section 7, that the multiplier $\alpha (1)$ 
is exceptional 
only at period two. 

We denote the set of periodic  points $p \in F$ with smallest period $n\in \Bbb N$ by $P^\circ_n$, and we  denote the set of points $p \in P^\circ_n$  with negative multiplier 
$\mu\in\mathcal M $ by $P^\circ_n(\mu_-).$
In the proofs of Sections 3, 4,   and  6, 7 we construct for the multiplier $\mu$ in question, and for a suitably chosen period $k\in \Bbb N$,  shift commuting injections
$$
\eta_n :P_n^\circ(\mu_-)
\rightarrow
 \bigcup_{\tilde\mu\in \mathcal M\setminus \{ \mu\}}
P^\circ_n(\tilde\mu_-), \quad n >k.
$$
This we do by first constructing a partition of  $P_n^\circ(\mu), n >k,$ (some  sets of which may be empty), together with a  shift commuting injection of each set of the partition into  $\bigcup_{\tilde\mu\in \mathcal M\setminus \{ \mu\}}
P^\circ_n(\tilde\mu_-),$ where we show injectivity on each set  by describing how a point  can be reconstructed from its image under $\eta_n$. Then we show that the  images under $\eta_n$  of the sets of the partition are disjoint.

%section 2
\section{Preliminaries}

We denote the set of admissible words of the Fibonacci-Dyck shift by $\mathcal L$. The empty word we denote by $\epsilon$, and  the length of a word we denote by $\ell$.

We denote by $\mathcal C(0)$ the circular code of words $c=(c_i)_{1\leq i \leq I}\in \mathcal L, I > 1$,  that begin with the symbol $\beta^-(0)$ and that are 
such that $\lambda(c) = \bold 1,$ and such that for no index $J,1 < J < I,$ one has that
$$
\prod _{1\leq j \leq J}\lambda (c_j) = \bold 1.
$$
Also we denote by $\mathcal C(1)$ the circular code \cite [Section 7]{BPR} of words $c=(c_i)_{1\leq i \leq I}\in \mathcal L), I > 3$, that begin with the symbol $\beta^-$ and end with the word $\beta^+(1)\beta^+$, and that are such that $\lambda(c) = \bold 1$, 
and such that for no index $J,1 < J < I,$ one has that $  c_J=  \beta^+ $, and 
$$
\prod _{1\leq j \leq J}\lambda (c_j) = \bold 1.
$$
We also have the circular codes
$$
\mathcal C = \mathcal C(0) \cup  \{ \beta^- \beta^+ \} \cup \mathcal C(1),
$$
and 
$$
 \mathcal C^\circ(1) = \beta^-(1)\mathcal C^\ast  \beta^+(1).
$$
Note that
$$
\mathcal C(0) =  \beta^-(0)\mathcal C^\ast \beta^+(0),
\qquad
\mathcal C(1) =  \beta^-( \mathcal C^\circ(1)^\ast \setminus \{ \epsilon \})\beta^+.
$$

A bijection
$$
\Psi_\circ: \mathcal C^\circ(1)   \to \mathcal C(0)
$$
is given by
$$
\Psi_\circ(\beta^-(1)f\beta^+(1) )= \beta^-(0)f\beta^+(0) , \quad f\in \mathcal C^\ast.
$$
The bijection $\Psi_\circ: \mathcal C^\circ(1)   \to \mathcal C(0)$ extends to a bijection
$$
\Psi: \mathcal C^\circ(1)^\ast   \to \mathcal C(0)^\ast
$$
by
$$
\Psi((c^\circ_k)_{1\leq k\leq K})=
 ((   \Psi_\circ ( c^\circ_k ))_{1\leq k\leq K}, \qquad c^\circ_k\in\mathcal C^\circ(1),
  1\leq k\leq K, \quad K \in \Bbb N.
$$

We set
$$
\mathcal B(1)= \beta^- \mathcal C^\circ(1)^\ast \beta^-(1),
$$
and we define a bijection
$$
\Xi:\mathcal B(1) \rightarrow\mathcal C(0),
$$
by
$$
\Xi( \beta^-f^\circ \beta^-(1)  ) =\beta^-(0)\Psi( f^\circ  )    \beta^+(0), \qquad f^\circ \in
\mathcal C^\circ(1)^\ast .
$$

We also set
$$
\mathcal B(0,0)= \beta^-(0) \mathcal C^\ast \beta^-(0),
$$
and we define a bijection
$$
\Phi_0:\mathcal C(0)  \rightarrow \mathcal B(0,0),
$$
by
$$
\Phi_0(\beta^-(0) f  \beta^+(0)   ) = \beta^-(0) f  \beta^-(0), \qquad  f\in  \mathcal C^\ast.
$$
We also set
$$
\mathcal B(1,1)= \mathcal B(1) \mathcal C^\ast \beta^-\beta^-(1),
$$
and we define an bijection
$$
\Phi_1: \mathcal C(1)  \rightarrow  \mathcal B(1,1),
$$
by
$$
\Phi_1( \beta^-f^\circ\beta^-(1)f\beta^+(1)  \beta^+ ) =  \beta^-f^\circ\beta^-(1)f
  \beta^- \beta^-(1),  \quad f^\circ \in \mathcal C^\circ(1)^\ast,    f\in \mathcal C^\ast .
$$

We set
$$
\mathcal Q_0 = 
(\mathcal C(0) \cup  \{\beta^- \beta^+ \} )^\ast\setminus  \{\beta^- \beta^+ \}^\ast,
$$
and we  define an injection
$$
\Delta_0: \mathcal Q_0\rightarrow   \mathcal L.
$$
For this we let $f \in  \mathcal Q_0$, 
$$
f =(c_k)_{1\leq k\leq K}, \qquad c_k \in \mathcal C(0) \cup  \{\beta^- \beta^+ \} ,1\leq k\leq K,\quad K \in \Bbb N,
$$
set
$$
k_\circ (f) = \max \{k \in [1, K]: c_k \in \mathcal C(0)  \},
$$
and set
$$
\Delta_0(f) = ((c_k)_{1 \leq k <k_{\circ}(f) }   , \Phi_0  ( c_{k_{\circ}(f)}    ) ,
(c_k)_{k_{\circ}(f)<k\leq K }).
$$

We also set
$$
\mathcal Q_1 =  \mathcal C^\ast
\setminus (  \mathcal C(0) \cup\{\beta^- \beta^+ \})^\ast,
$$
and  define an injection
$$
\Delta_1: \mathcal Q_1\rightarrow   \mathcal L
$$
For this we let $f \in  \mathcal Q_1$, 
$$
f =(c_k) _{1\leq k\leq K}, \qquad c_k \in 
 \mathcal C,
1\leq k\leq K,\quad K > 1,
$$
set
$$
k_\circ (f) = \max \{k \in [1, K]: c_k \in \mathcal C(1)  \},
$$
and set
$$
\Delta_1(f) = ((c_k)_{1 \leq k <k_{\circ}(f) }   , \Phi_1  ( c_{k_{\circ}(f)}    ) ,
(c_k)_{k_{\circ}(f)<k\leq K }).
$$

We put a linear order on the alphabet  of $F$. The resulting lexicgraphic order on 
$\mathcal L$  will be used to single out an element of $\Bbb Z/n \Bbb Z $,
when constructing the shift commuting maps
$$
\eta_n :P_n^\circ(\mu_-)
\rightarrow
 \bigcup_{\tilde\mu\in \mathcal M\setminus \{ \mu\}}
P^\circ_n(\tilde\mu_-), \quad n >k.
$$
If a word appears in a point $p \in P^\circ_n$ with its last symbol at index $i \in \Bbb Z$ then we say that the word appears at index  $i$. 
For $p \in F$ we denote by $\mathcal I^{(0)}(p) $
($   \mathcal I^{(1)}(p)$)
the set of indices $i\in \Bbb Z$ such that $p_i = \beta_-(0)$($ \beta_-(1) $) and
$$
\lambda (p_{[i, i+k]}) \neq \bold 1,\quad k \in  \Bbb N.
$$
We say that a word appears openly in $p \in F$ if it appears at an index $i\in \mathcal I^{(0)}(p)
\cup      \mathcal I^{(1)}(p)$.

For an element $\gamma $ of the free monoid that is generated by $\alpha(0)$ and $\alpha(1)$ (or by $\alpha_-(0)$ and $\alpha_-(1)$), e.g. for
$$
\gamma = \prod_{0\leq n < N}\alpha(0)^{K(0,n)}\alpha(1)^{K(1,n)}, \qquad 
K(0,n),K(1,n) \in \Bbb Z_+,  \ 0\leq n < N,
$$
we use the notation
$$
\nu_0 ( \gamma) =\sum_{0\leq n < N}K(0,n), \quad\nu_1 ( \gamma) =\sum_{0\leq n < N}K(1,n),
$$
and, choosing for $\gamma$ any representative of the multiplier $ \mu  \in \mathcal M,$ we set
$$
\nu_0 (\mu) =  \nu_0 ( \gamma)  , \quad \nu_1(\mu) = \nu_1(\gamma).
$$
A point
$
p \in P^\circ_n(\mu_-), \mu \in \mathcal M,
$
determines a $\kappa_p\in \Bbb N$ by
$$
(\nu_0(\lambda(p_{[0, n)})),\nu_1(\lambda(p_{[0, n)}))   ) = \kappa_p(\nu_0(\mu),\nu_1(\mu   )). 
$$

For  
 $$
 b =  \beta^-f^\circ \beta^-(1)\in \mathcal B{(1)} ,
 $$
we set 
 $$
 \Lambda (b)= \ell(f^\circ ),
 $$
 and for
 $$
 b =  \beta^-(0)f \beta^-(0)\in \mathcal B{(0,0)} ,
 $$
 we set
 $$
  \Lambda (b)= \ell(f).
  $$

We also define a subset  $\mathcal D{(1,1)}$ of $\mathcal B{(1)}\mathcal C^\ast\mathcal B{(1)}$ by
\begin{align*}
\mathcal D{(1,1)} =
 \{ \beta^-f^{\circ,-}  \beta^-(1)f \beta^-f^{\circ,+}  \beta^-(1)
\in \beta^-\mathcal C^\circ(1)^\ast  \beta^-(1)& \mathcal C^\ast  \beta^-(0)\mathcal C^\circ(1)^\ast  \beta^-(1): \\
& \ell(  f)   \geq  \ell(f^{\circ,-} )  , \ell( f^{\circ,+} ) \},
\end{align*}
and a subset  $\mathcal D{(0,1)}$ of $ \beta^-(0)\mathcal C^\ast\mathcal B{(1)}$ by
\begin{align*}
\mathcal D{(0,1)} = \{  \beta^-(0)f \beta^-f^\circ \beta^-(1)\in\beta^-(0)\mathcal C^\ast\beta^-\mathcal C^\circ(1)^\ast \beta^-(1):  \ell ( f  )\geq \ell ( f^\circ )\},
\end{align*}
as well as a subset  $\mathcal D{(1,0)}$ of $\mathcal B{(1)}\mathcal C^\ast \beta^-(0)$ by
\begin{align*}
\mathcal D{(1,0)} =\{ \beta^-f^\circ \beta^-(1)f \beta^-(0) \in\beta^-\mathcal C^\circ(1)^\ast \beta^-(1)\mathcal C^\ast\beta^-(0): \ell ( f  )\geq \ell ( f^\circ )\},
\end{align*}
and for
  $$
 d =  \beta^-f^{\circ,-}  \beta^-(1)f \beta^-f^{\circ,+}  \beta^-(1) \in \mathcal D{(1,1)} ,
 $$
 and 
 $$
 d =  \beta^-(0)f \beta^-f^\circ \beta^-(1) \in \mathcal D{(0,1)},
 $$
 and
 $$
 d = \beta^-f^\circ \beta^-(1)f \beta^-(0) \in \mathcal D{(1,0)} ,
 $$
we set 
  $$
 \Lambda (d)= \ell(f).
 $$
 
For a point $p\in P_n^\circ(F)$ we denote by $\Lambda (p)$ be the maximal value of $\Lambda(d)$ of words 
 $$
 d \in\mathcal B{(0,0)} \cup  \mathcal B{(1)} \cup\mathcal D{(1,1)} \cup \mathcal D{(0,1)} \cup \mathcal D{(1,0)}
 $$
 that appear openly in $p$, and we denote by   $\mathcal J^{(0,0)} (p)$$(\mathcal J^{(1)}   (p),\mathcal J^{(1,1)} (p), 
 \mathcal J^{(0,1)} (p)$,  $\mathcal J^{(1,0)} (p)$)
  the set of  indices, at which there appears in
   $p$ openly a word $d\in \mathcal B{(0,0)} $ ($d\in \mathcal B{(1)}, d \in\mathcal D{(1,1)},d\in\mathcal D{(0,1)},d \in\mathcal D{(1,0)}$)  such that  $\Lambda (p)=\Lambda (d)$, and we 
denote   by  $\mathcal J_\circ^{(0,0)} (p)$($  \mathcal J_\circ ^{(1)}   (p),
 \mathcal J_\circ ^{(1,1)} (p), 
  \mathcal J_\circ^{(0,1)} (p)$,  $\mathcal J_\circ^{(1,0)} (p)$)
the set of   indices $j_\circ \in \mathcal J^{(0,0)}(p)$ $(\mathcal J^{(1)}   (p),\mathcal J^{(1,1)} (p), 
 \mathcal J^{(0,1)} (p)$,  $\mathcal J^{(1,0)} (p)$) such that the word 
 $p_{(j_\circ -n, j_\circ ]}$ is lexicographically  \ the smallest  \ one among \ the words \  
$
p_{(j-n, j]}, j \in  \ \mathcal J^{(0,0)}(p)(\ \mathcal J^{(1)}  (p),\ \mathcal J^{(1,1)} (p),$ $ 
\mathcal J^{(0,1)} (p),  \mathcal J^{(1,0)} (p)).
$

%section 3
\section{The multiplier $\alpha(0)\alpha(1)$}

\begin{lemma1} 
The multiplier $\alpha(0)\alpha(1)$  is not exceptional for the Fibonacci-Dyck shift.
\end{lemma1}

\begin{proof} 
By Lemma (b) the multiplier $\alpha(0)\alpha(1)$ is not exceptional for the Fibonacci-Dyck shift for periods three and five.

We construct shift commuting injections
$$
\eta_n  :P_n^\circ(\alpha_-(0)\alpha_-(1))
\rightarrow
 \bigcup_{\tilde\mu\in \mathcal M\setminus \{\alpha(0)\alpha(1)\}}
P^\circ_n(\tilde\mu_-),    \quad n > 5.
$$

Let $m> 2$.
Let $P^{(0)}_{2m+1}$ be the set of $ p\in P^\circ_{2m+1}(\alpha_-(0)\alpha_-(1)) $, such that $\kappa_p= 1$,
which means that 
$$
p_{(i-2m-1, i ]} \in\mathcal C^\ast\beta^-(0)\mathcal C^\ast\beta^-\mathcal C^\circ(1)^\ast\beta^-(1),
 \quad i \in  \mathcal I^{(1)}(p),
$$
and for  $p\in P^{(0)}_{2m+1}$   let the words 
$$
  f^-(p),
f^+(p)\in \mathcal C^\ast,  \qquad  f^\circ(p)\in \mathcal C^\circ(1)^\ast, 
 $$
be given by writing
$$
p_{(i-2m-1, i ]} = f^-(p)\beta^-(0)f^+(p)\beta^- f^\circ(p)\beta^-(1), \quad i \in \mathcal I^{(1)}
(p).
$$

We set
$$
P^{[1]}_{2m+1} = \{p \in P^{(0)}_{2m+1}:  f^+(p) \in \mathcal Q_1   \}.
$$

The shift commuting map $\eta_{2m+1}$ is to map a point $p\in P^{[1]}_{2m+1}$ to the point 
$q \in P^\circ_{2m+1}$
that is given by
$$
q_{(i-2m-1, i ]} =f^-(p)\beta^-(0)\Delta_1(f^+(p)) \beta^- f^\circ(p)\beta^-(1), \qquad i \in \mathcal I^{(1)}(p).
$$

For $q\in  \eta_{2m+1}(P^{[1]}_{2m+1})$ one has
$$
q_{(i-2m-1, i ]}\in \mathcal C^\ast\mathcal B(1,1)\mathcal C^\ast\mathcal B(1) \mathcal C^\ast
\beta^-(0), \qquad  i \in \mathcal I^{(0)}(q).
$$  
and with the words
$$
g(q), g^{-}(q)\in \mathcal C^\ast,  g^{+}(q) \in (\mathcal C(0)\cup \{\beta^-\beta^+ \})^\ast,
  \quad b(q)\in\mathcal B(1,1),h(q)\in\mathcal B(1),
  $$
that are given by
$$
q_{(i-2m-1, i ]}=   g^{-}(q)b(q)g^{+}(q)h(q)g(q) \beta^-(0), \quad i \in \mathcal I^{(0)}(q),
$$
the point  $p\in P^{[0]}_{2m+1}$ can be reconstructed from its image $q$ under $\eta_{2m+1}$ as the point in $ P^{(0)}_{2m+1}$ that is given by
$$
p_{(i-2m-1, i ]}=  g^{-}(q)\Phi_1^{-1}(b(q))g^{+}(q)h(q)g(q) \beta^-(0), \quad i \in \mathcal  I^{(0)}(q).
$$ 

We note that
  \begin{multline*}
  \nu_0(\lambda( \eta_{2m+1}(p)_{[0, 2m+1)}) )= 1, \quad 
   \nu_1(\lambda( \eta_{2m+1}(p)_{[0, 2m+1)}))  = 3,    \tag {P.01[1]}\\
  p \in P^{[1]}_{2m+1}.
 \end{multline*}

We set
$$
P^{[0]}_{2m+1} = \{p \in P^{(0)}_{2m+1}:  f^+(p) \in
 \mathcal Q_0   \}.
$$

The shift commuting map $\eta_{2m+1}$ is to map a point $p\in P^{[0]}_{2m+1}$ to the point 
$q \in P^\circ_{2m+1}(F)$
that is given by
$$
q_{(i-2m-1, i ]} =f^-(p)\beta^-(0)\Delta_0(f^+(p)) \beta^- f^\circ(p)\beta^-(1), \qquad i \in \mathcal I^{(1)}(p).
$$

For $q\in  \eta_{2m+1}(P^{[0]}_{2m+1})$ one has
$$
q_{(i-2m-1, i ]}\in \mathcal C^\ast\beta^-(0)\mathcal B(0,0)
\mathcal C^\ast\beta^- \mathcal C^\circ(1)^\ast\beta^-(1), \qquad  i \in \mathcal I^{(1)}(q).
$$  
and with the words
$$
 g(q)\in  \mathcal C^\ast,g^{-}(q)\in  (\mathcal C(0)\cup \{\beta^-\beta^+ \})^\ast ,g^+(q) 
 \in\{\beta^-\beta^+ \})^\ast, 
 $$
 $$
 b(q) \in \mathcal B(0,0),
   g^\circ(q)\in\mathcal C^\circ(1)^\ast,
  $$
that are given by
$$
q_{(i-2m-1, i ]}= g(q)\beta^-(0)g^{-}(q)b(q)g^+(q)
\beta^- g^\circ(q)\beta^-(1), \quad i \in \mathcal I^{(1)}(q),
$$
the point  $p\in P^{[0]}_{2m+1}$ can be reconstructed from its image $q$ under $\eta_{2m+1}$ as the point in $ P^{(0)}_{2m+1}$ that is given by
$$
p_{(i-2m-1, i ]}=   g(q)\beta^-(0)g^{-}(q)\Phi_0^{-1}(b(q))g^+(q)
\beta^- g^\circ(q)\beta^-(1),\quad i \in \mathcal  I^{(1)}(q).
$$ 

We note that
  \begin{multline*}
  \nu_0(\lambda(\eta_{2m+1}(p)_{[0, 2m+1)}) )= 3,  \quad
   \nu_1(\lambda(\eta_{2m+1}(p)_{[0, 2m+1)}))  = 1,      \tag {P.01[0]}  \\
  p \in P^{[1]}_{2m+1}.
 \end{multline*}

We set
$$
P^{[\beta]}_{2m+1} =  P^{(0)}_{2m+1}
\setminus ( P^{[1]}_{2m+1}
\cup P^{[0]}_{2m+1} ).
$$

The shift commuting map $\eta_{2m+1}$ is to map a point $p\in P^{[\beta]}_{2m+1}$ to the point 
$q \in P^\circ_{2m+1}(F)$
that is given by
$$
q_{(i-2m-1, i ]}= f^-(p)\beta^-(0)f^+(p)\Xi( \beta^-f^\circ(p) \beta^-(1)), \quad i \in   \mathcal I^{(1)}(p).
$$

With words
$$
 g^{(\beta)}(q) \in \{ \beta^-\beta^+ \}^\ast, \quad c(q) \in \mathcal C(0), 
 \quad g(q) \in \mathcal C^\ast,
$$
that are  given by
\begin{align*}
q_{(i-2m-1, i ]}= 
g^{(\beta)}(q)c(q)g(q)\beta^-(0),
 \quad i \in \mathcal I^{(0)}(q),
\end{align*}
a point  $p\in P^{[\beta]}_{2m+1}$ can be reconstructed from its image $q$ under $\eta_{2m+1}$ as the point  $p\in P^{(\circ)}_{2m+1}$, that is given by
\begin{align*}
p_{(i-2m-1, i ]}= 
g^{(\beta)}(q)  \Xi^{-1}  (c(q))   g(q)\beta^-(0),
 \quad i \in \mathcal I^{(0)}(q).
\end{align*}

We note that
  \begin{multline*}
  \nu_0(\lambda(\eta_{2m+1}(p)_{[0, 2m+1)}) )= 1, \quad
   \nu_1(\lambda(\eta_{2m+1}(p)_{[0, 2m+1)}))  = 0, \tag {P.01$[\beta]$}
 \\
 p \in P^{[\beta]}_{2m+1}.
   \end{multline*}
  
We set
$$
P^{(0)}_{2m} = \emptyset, \quad m > 3,
$$
and for $n> 5$ we set
$$
P^{(1)}_n = \{p\in P^\circ_n(\alpha_-(0)\alpha_-(1)) \setminus P^{(0)}_n:  \mathcal J^{(1)}  (p) \neq \emptyset \}.
$$
The shift commuting map $\eta_n$ is to map a point $p \in P^{(1)}_n $ to the point $q \in P_n(F)$
that is obtained by replacing in the point $p$ each of the words $b\in \mathcal B{(1)}$  
that appear at the indices in $\mathcal J_\circ^{(1)}(p)$ by the word
$
\Xi (b).
$

A point $p\in P^{(1)}_n$ can be reconstructed from its image  $q$ under $\eta_n$ by replacing in $q$ the word $c(q) \in  \mathcal C(0) $, that
 is identified as the unique word in $\mathcal C^\ast$ of maximal length that appears in $q$, by the word $\Xi^{-1}(c(q)).$
  
  We note that
   \begin{align*}
(\nu_0(\lambda(\eta_n(p)_{[0, n)})  ,\nu_1(\lambda(\eta_n(p)_{[0, n)}) )=   (\kappa_p ,\kappa_p   - 1 ) , \quad p \in  P^{(1)}_n .\tag {P.01.1}
\end{align*}
  
We set
$$
P^{(0,1)}_n = \{p\in P^\circ_n(\alpha_-(0)\alpha_-(1))\setminus ( P^{(0)}_n\cup P^{(1)}_n  )
:  \mathcal J^{(0,1)}  (p) \neq \emptyset \}.
$$

The shift commuting map $\eta_n$ is to map a point  $p \in P^{(0,1)}_n $ to the point $q \in P_n$
that is obtained by replacing in the point $p$  the words $b\in\mathcal B(1)$ that appear in $p$ at the indices in
 $\mathcal J_\circ^{(0,1)}(p)$ by the word
$
\Phi_0(\Xi(b)).
$

A point $p\in P^{(0,1)}_n$ can be reconstructed from its image $q$ under $\eta_n$ by replacing in $q$ the word 
$$
\beta^-(0)h(q)b(q)\in \beta^-(0)  \mathcal C^\ast\mathcal B(0,0),
$$
whose prefix $ \beta^-(0)h(q) \beta^-(0) $ is identified as the unique word in $ \mathcal B(0,0)$ of maximal length, that appears openly in $q$, by the word
$$
 \beta^-(0)h(q)  \Xi^{-1}(\Phi_0^{-1}(b(q))). 
 $$

We note that 
\begin{align*}
(\nu_0(\lambda(\eta_n(p)_{[0, n)})  ,\nu_1(\lambda(\eta_n(p)_{[0, n)}) )= 
 (\kappa_p +2,\kappa_p   -1  ), \quad
 p \in P^{(0,1)}_n .\tag {P.01.01}
\end{align*}

We set
$$
P^{(1,0)}_n =  P^\circ_n(\alpha^-(0)\alpha^-(1))\setminus ( P^{(0)}_n\cup P^{(1)}_n\cup P^{(0,1)}_n   ).
$$

With the  words 
$
 b(p) \in \mathcal B(1),
  f(p) \in \mathcal C^\ast
$ 
that are given by writing the  word
in $\mathcal D{(1,0)}$ 
that appears  at the indices 
in $\mathcal J_\circ^{(1,0)}(p)$ as
$
b(p) f(p) \beta^-(0) ,
$
the shift commuting map $\eta_n$ is to map a point $p \in P^{(1,0)}_n $ to the point $q \in P_n(F)$
that is obtained by replacing in the point $p$  the words in $\mathcal D{(1,0)}$,
 that appear at the indices in $\mathcal J_\circ^{(1,0)}(p)$ by the word
$
\Phi_0(\Xi(b(p)))
 f(p)\beta^-(0).
$

A point $p \in P^{(1,0)}_n $ can be reconstructed from its image $q$ under $\eta_n$ by replacing in $q$ the word  
$$ 
  b(q)\beta^-(0) h (q)\beta^-(0)  \in
 \mathcal B(0,0)\mathcal C ^\ast\beta^-(0), 
 $$  
 whose suffix $ \beta^-(0) h (q)\beta^-(0))$ is identified as the unique word of maximal length in 
 $\mathcal B(0,0)$, that  appears openly  in $q$,
   by the word  
  $$
 \Xi^{-1}(\Phi_0^{-1}(b(q)))h(q) \beta^-(0).
  $$
   
We note that 
\begin{align*}
(\nu_0(\lambda(\eta(p)_{[0, n)})  ,\nu_1(\lambda(\eta(p)_{[0, n)}) )=  (\kappa_p +2,\kappa_p   -1 ), \quad
 p \in  \eta_n(P^{(1,0)}_n ).\tag {P.01.10}
\end{align*}
  
We have produced a partition 
\begin{align*}
P^\circ_n(\alpha(0)\alpha(0))  = P^{(0)}_n \cup P^{(1)}_n\cup P^{(0,1)}_n\cup P^{(1,0)}_n. \tag {P.01}
\end{align*}

In points  $q \in \eta_n( P^{(0,1)}_n)$ the unique word in $\beta^-(0)\mathcal C^\ast\beta^-(0)$ of maximal length that appears openly  in $q$ is followed by a word in  $\mathcal C^\ast\ \beta^-(0)$, whereas  in points  $q \in \eta_n( P^{(1,0)}_n  )$ the unique  word in $\beta^-(0)
\mathcal C^\ast\beta^-(0)$ of maximal length that appears openly  in $q$ is followed by a word in  $\mathcal C^\ast\ \beta^-$. From this observation and from (P.01[0]), (P.01[1]), (P.01[$\beta$]) and  
(P.01.1),  (P.01.01),  (P.01.10)
it follows that the 
images under $\eta_n$ of the sets of the partition (P.01)
 are disjoint. From (P.01[0]), (P.01[1]), (P.01[$\beta$]) and  
(P.01.1),  (P.01.01),  (P.01.10) it follows  also that
 $$
 \eta_n(P^\circ_n(\alpha_-(0)\alpha_-(1)) \cap  P^\circ_n(\alpha_-(0)) = \emptyset.
 $$
 
We have shown that
$$
\card ( \mathcal O_n(\alpha_-(0)\alpha_-(1))) \leq
 \card ( \bigcup_{\widetilde\mu\in
  \mathcal M \setminus \{ \alpha(0)\alpha(1) \} }
\mathcal O_n(\widetilde\mu^-).
$$
Apply Lemma (c).
\end{proof}

%section 4
\section{The remaining multipliers}

\begin{lemma2} 
Besides the multipliers
$\alpha (0)$ and $ \alpha (1)$
the Fibonacci-Dyck shift has no exceptional multipliers.
\end{lemma2}

\begin{proof} 
Consider a multiplier
 \begin{align*}
 \mu \not \in \{\alpha (0), \alpha (1), \alpha (0) \alpha (1)\}, 
 \end{align*}
  of $F$.
We construct shift commuting injections
$$
\eta_{n}  :P_{n}^\circ(\mu_-)
\rightarrow
 \bigcup_{\tilde\mu\in \mathcal M\setminus \{\mu\}}
P^\circ_{n}(\tilde\mu_-),    \quad n > 4.
$$

\indent
We set
$$
P^{(0,0)}_n = \{p\in P^\circ_n(\mu) :  \mathcal J^{(0,0)}  (p) \neq \emptyset \}.
$$
 
With the word 
$b(p) \in  \mathcal B{(0,0)},$
that appears in $p$  at the indices in $ \mathcal J_\circ^{(0,0)}(p)$, 
the shift commuting map $\eta_n$ is to map a point $p \in P^{(0,0)}_n $ to the point $q \in P_n$,
that is obtained by replacing in  $p$   the words $b(p)$ in $  \mathcal B{(0,0)}$, that appear in $p$  at the indices in $ \mathcal J_\circ^{(0,0)}(p)$
 by the word
$
\Phi_0^{-1}(b(p)).
$

The point $p$ can be reconstructed from its image  $q$ under $\eta_n$ by replacing in $q$ the word 
$
c(q),
 $
that is identified as the unique word in $\mathcal C(0)$ of maximal length, that appears in $q$,  by the word $ \Phi_0(c(q) ).$
 
  We note that
 \begin{align*}
( \nu_0(\lambda(\eta_n(p)_{[0, n)}), \nu_1(\lambda(\eta_n(p)_{[0, n)}))) =
(\kappa_p  \nu_0(\mu)-2   ,  \kappa_p  \nu_1(\mu)    ), \quad p \in P^{(0,0)}_n. \tag {0.0}
\end{align*}

We set
$$
P^{(1,1)}_n = \{p\in P^\circ_n(\mu)\setminus P^{(0,0)}_n :  \mathcal J^{(1,1)}  (p) \neq \emptyset \}.
$$

With the words
$$
b(p), \in \mathcal B(1), \quad f(p) \in  \mathcal C^\ast   , \quad f^{\circ}(p) \in  \mathcal C^\circ(1)^\ast,
$$
that are given by writing the word in $\mathcal D(1,1)$,
that appears  in $p$ at the indices in $ \mathcal J_\circ^{(1,1)}(p)$, as
$$
b(p)  f(p)  \beta^-f^{\circ} (p) \beta^-(1),
$$
the shift commuting map $\eta_n$ is to map a point $p \in P^{(1,1)}_n $ to the point $q \in P_n(F)$,
that is obtained by replacing in  $p$  the words in $\mathcal D(1,1)$,
that appear  in $p$ at the indices in $ \mathcal J_\circ^{(1,1)}(p)$,
 by the word
$$
\Phi_0(\Xi(b(p)) f(p)\beta^+(0) \Psi(f^{\circ}(p) ) \beta^-(0).
$$

A point $p \in P^{(1,1)}_n $ can be reconstructed from its image  $q$ under $\eta_n$ by replacing in $q$ the word
$$
\beta^-(0)h^-(q) \beta^-(0)h(q)\beta^+(0) h^+(q) \beta^-(0),
  $$ 
 that is identified as the word with the uniquely determined word  $\beta^-(0)h(q)\beta^+(0) \in \mathcal C(0)$ of maximal length, that appears in $q$, as  infix,  a uniquely determined, openly in $q$ appearing word $ h^+(q) \beta^-(0) \in \mathcal C(0)^\ast\beta^-(0)$, as suffix, and  a uniquely determined word
 $\beta^-(0)h^-(q) \in\beta^-(0)\mathcal C(0)^\ast$ as prefix,  
   by the word 
$$
\beta^-\Psi^{-1}(h^-(q) )\beta^-(1)  h(q)   \beta^- \Psi^{-1}(h^+(q) ) \beta^-(0).
$$

We note that
\begin{multline*}
( \nu_0(\lambda(\eta_n(p)_{[0, n)})), \nu_1(\lambda(\eta_n(p)_{[0, n)}))) =
(\kappa_p  \nu_0(\mu) +2  ,  \kappa_p  \nu_1(\mu)  -2  ), \tag{1.1}\\
 p \in P^{(1,1)}_n.
\end{multline*}

We denote by
$P^{(0,1,0)}_n$
the set of points in 
$$  P^\circ_n(\mu)\setminus( P^{(0,0)}_n\cup P^{(1,1)}_n   )
$$ 
such that $ \mathcal J^{(1)}  (p) \neq \emptyset $, and such that the word in $\mathcal B{(1)}$, that appears at the indices in $ \mathcal J_\circ^{(1)}(p)$,  is preceded in $p$ by a word in 
$ \beta^-(0)\mathcal C^\ast $, and followed in $p$ by a word in $\mathcal C^\ast  \beta^-(0) $.

With the word
$ f(p) \in \mathcal C^\ast,$
that is given by writing 
the  openly appearing word in 
$\mathcal C^\ast\beta^-(0)$, that follows
the word  $b(p) \in\mathcal B{(1)}$ that appears  in $p$ at the indices in $ \mathcal J_\circ^{(1)}(p)$, as  $f(p)\beta^-(0)$,
the shift commuting map $\eta_n$ is to map a point $p \in P^{(0,1,0)}_n $ to the point $q \in P_n(F)$
that is obtained by replacing in the point $p$  the words in $\mathcal B{(1)}$ that appear  in $p$ at the indices in  $ \mathcal J^{(1)}_\circ(p)$, together with the
  openly appearing words in 
$\mathcal C^\ast\beta^-(0)$, that follow them,
 by the word
$
\Xi(b(p))f(p)\beta^+(0).
$

Denoting for a point $p \in P^{(0,1,1)}_n $  by $q^\prime \in F$ the point, that is obtained from its image $q$ under $\eta_n$ by replacing in $q$ the unique word 
$c(q)\in \mathcal C(0)$ of maximal length, that appears in $q$,  by the word 
$\Phi(c(q))$, one sees, that
the point $p $ can be reconstructed from $q$  by replacing in the point $q^\prime$ the unique word $c(q^\prime )\in \mathcal C(0)$ of maximal length, that appears in $q^\prime $, by the word $\Xi^{-1}(c(q^\prime))$.

 We note that
 \begin{multline*}
( \nu_0(\lambda(\eta_n(p)_{[0, n)})), \nu_1(\lambda(\eta_n(p)_{[0, n)}))) = 
(\kappa_p  \nu_0(\mu) -2  ,  \kappa_p  \nu_1(\mu)  -1  ),  \tag {0.1.0} \\  p\in P^{(0,1,0)}_n.
\end{multline*}

We denote by
$P^{(\bullet,1,1)}_n$ 
the set of points $p$ in 
$$ 
 P^\circ_n(\mu_-)\setminus (P^{(0,0)}_n \cup P^{(0,1,0)}_n)
$$ 
such that $ \mathcal J^{(1)}  (p) \neq \emptyset $, and such that the words in $\mathcal B{(1)}$ that appear  at the indices in $ \mathcal J_\circ^{(1)}(p)$  are followed in $p$ by a word in 
$\mathcal C^\ast \beta^-\mathcal C^\circ(1)^\ast \beta^-(1) $.

The shift commuting map $\eta_n$ is to map a point $p \in P^{(\bullet,1,1)}_n $ to the point $q \in P_n(F)$
that is obtained by replacing in  $p$  the words 
$b(p) \in \mathcal B{(1)}$, that appear in $p$ at the indices in $ \mathcal J_\circ^{(1)}(p)$, by the word
$\Phi_0(\Xi(b(p))$.

A point $p \in P^{(\bullet,1,1)}_n$ can be reconstructed from its image  $q$ under $\eta_n$ by replacing in $q$ the word 
$
h(q)
  $ 
 that is identified as the unique word in $\mathcal B(0,0)$ of maximal length that appears in $q$,  by the word $ \Xi  ^{-1}(\Phi_0^{-1}(h(q)))$.

 We note that
 \begin{multline*}
( \nu_0(\lambda(\eta_n(p)_{[0, n)}), \nu_1(\lambda(\eta_n(p)_{[0, n)})= 
(\kappa_p  \nu_0(\mu) +2  ,  \kappa_p  \nu_1(\mu)  -1  ),  \tag {$\bullet$.1.1} \\ p \in P^{(\bullet,1,1)}_n.
\end{multline*}

We denote by
$P^{(1,1,\bullet)}_n$
the set of points in 
$$ 
 P^\circ_n(\mu_-)\setminus (P^{(0,0)}_n \cup P^{(0,1,0)}_n\cup P^{(\bullet,1,1)}_n)
$$ 
such that $ \mathcal J^{(1)}  (p) \neq \emptyset $, and such that the word in $\mathcal B{(1)}$ that appears  at indices in $ \mathcal J_\circ^{(1)}(p)$,  is preceded in $p$ by a word in 
$\beta^-\mathcal C^\circ(1)^\ast \beta^-(1)\mathcal C^\ast $.

The shift commuting map $\eta_n$ is to map a point $p \in P^{(1,1,\bullet)}_n $ to the point $q \in P_n(F)$
that is obtained by replacing in  $p$  the words  $b(p) \in \mathcal B{(1)}$, that appear in 
$p$ at the indices in $ \mathcal J_\circ^{(1)}(p)$, by the word
$\Xi(b(p))$.

A point $p\in P^{(1,1,\bullet)}_n$ can be reconstructed from its image  $q$ under $\eta_n$ by replacing in $q$ the word 
$ 
c(q)
  $
 that is identified as the unique word in $\mathcal C(0)$ of maximal length, that appears in $q$, by the word $\Xi^{-1}(c(q) ) .$

We note that
 \begin{align*}
(\nu_0(\lambda(\eta_n(p)_{[0, n)}), \nu_1(\lambda(\eta_n(p)_{[0, n)}) =
(\kappa_p  \nu_0(\mu)   ,  \kappa_p  \nu_1(\mu)  -1  ),     \tag {1.1.$\bullet$} 
\quad  p \in P^{(1,1,\bullet)}_n.
\end{align*}

We set
$$
P^{(0,1)}_n =\{p\in P^\circ_n(\mu_-)\setminus (P^{(0,0)}_n \cup P^{(0,1,0)}_n
\cup P^{(\bullet,1,1)}_n)\cup P^{(1,1,\bullet)}):   \mathcal J^{(1,0)}(p) \neq \emptyset \}.
$$

 The shift commuting map $\eta_n$ is to map a point $p \in P^{(0,1)}_n $ to the point $q \in P_n(F)$,
that is obtained by replacing in  $p$  the words $b(p) \in  \mathcal B{(1)}$
  that appear at the indices in $\mathcal J_\circ^{(0,1)}(p)$, by the word
  $\Phi_0(\Xi(b(p)))$.

The point $p\in P^{(0,1)}_n$ can be reconstructed from its image  $q$ under $\eta_n$ by replacing in $q$ the word 
$
\beta^-(0)g(q)\beta^-(0)
  $
 that is identified as the unique word in $\mathcal B(0,0)$ of maximal length, that appears in $q$,  
 by the word
$\beta^-(0)g(q) \beta^-$,
and the open appearances $h(q)\beta^-(0)$ of a word in $\mathcal C(0)^\ast \beta^-(0)$, that follow in $q$ the word 
$\beta^-(0)g(q)\beta^-(0)$,
 by the word $\Psi^{-1}(h(q) ) \beta^-(1).$   
  
We note that
 \begin{align*}
( \nu_0(\lambda(\eta(p)_{[0, n)})), \nu_1(\lambda(\eta(p)_{[0, n)}))) = 
(\kappa_p  \nu_0(\mu) +2  ,  \kappa_p  \nu_1(\mu)  -1  ),    \tag{0.1}   \ \  p \in P^{(0,1)}_n.
\end{align*}

We set
$$
P^{(1,0)}_n = P^\circ_n(\mu_-)\setminus (P^{(0,0)}_n \cup P^{(0,1,0)}_n\cup P^{(\bullet,1,1)
}_n\cup P^{(1,1,\bullet)}\cup P^{(0,1)}_n).
$$
 
With the words
$$
 b(p) \in \mathcal B, \quad f(p) \in \mathcal C^\ast,
 $$
that are given by writing the word in $ \mathcal D{(1,0)}$,
  that appears at the indices in $\mathcal J_\circ^{(1,0)}(p)$, as
  $$
  b(p)f (p) \beta^-(0),
  $$
the shift commuting map $\eta_n$ is to map a point $p \in P^{(1,0)}_n $ to the point 
$q \in P_n(F)$,
that is obtained by replacing in  $p$  the words  $b(p)f (p) \beta^-(0)$,
  that appear at the indices in $\mathcal J_\circ^{(1,0)}(p)$, by the word 
$$
\Phi_0(\Xi(b(p)) f(p)  \beta^+(0).
$$

The point $p\in P^{(1,0)}_n$ can be reconstructed from its image  $q$ under $\eta_n$ by replacing in $q$ the word 
$\beta^-(0)g(q)\beta^+(0)$,
 that is identified as the unique word in $\mathcal C(0)$ of maximal length, that appears in $q$, by the word
$  \beta^-(1)g(q)\beta^-(0)$, and the word $\beta^-(0)h(q) \in \beta^-(0) \mathcal C^\ast$  that precedes the word  $\beta^-(0)g(q)\beta^+(0)$ in $q$, by the word
  $ \beta^- \Psi^{-1}(h(q) )$.
  
We note that
 \begin{align*}
( \nu_0(\lambda(\eta(p)_{[0, n)})), \nu_1(\lambda(\eta(p)_{[0, n)}))) = 
(\kappa_p  \nu_0(\mu)   ,  \kappa_p  \nu_1(\mu)  -1  ), \quad  p \in P^{(1,0)}_n.\tag {1.0}
\end{align*}

We have produced a partition
\begin{align*}
P^\circ_n(\mu_-)=  P^{(0,0)}_{n}  \cup P^{(0,1,0)}_{n} 
\cup P^{(\bullet,1,1)}_{n} \cup P^{(1,1,\bullet)}_{n}
\cup P^{(0,1)}_{n}\cup P^{(1,0)}_{n}. \tag {P}
\end{align*}

In a point $q\in\eta_{n}(P^{(\bullet,1,1)}_{n})$ the word in $\mathcal B(0,0)$ of maximal length, that appears in $q$, is followed  by an open appearance of a word in $\mathcal C^\ast\mathcal B{(1)}$,
   whereas in the points $q\in\eta_{n}(P^{(0,1)}_{n})$ this word is followed by an open 
   appearance of a word in $\mathcal C^\ast\beta^-(0)$. 
   Also, in a point $q\in\eta_{n}(P^{(1,1,\bullet)}_{n})$ the word in $\mathcal C(0)$ of maximal length, that appears in $q$, is preceded  by  a word in $\mathcal B{(1)}\mathcal C^\ast$,
   whereas in a point  $q\in\eta_{n}(P^{(1,0)}_{n})$ this word is preceded  by a word in 
   $\beta^-(0)\mathcal C^\ast$.
     
 It follows from these observations and from  
(00), (0,1,0), ($\bullet$,1,1), (1,1,$\bullet$), (0,1), (1,0),
 that the images under $\eta_{n}$ of the elements of the partition (P) are disjoint. From 
 (0,0), (0,1,0), ($\bullet$,1,1), (1,1,$\bullet$), (0,1), (1,0), it follows also, that
$$
\eta_n(P^\circ_n(\mu_-)) \cap  P^\circ_n(\mu_-)  = \emptyset.
$$

We have shown that
$$
\card \  \mathcal O_n(\mu_-) \leq \card ( \bigcup_{\widetilde{\mu}\in \mathcal M \setminus \{ \mu^- \}}  \mathcal O_n(\widetilde{\mu}_-)),
$$
The lemma follows now from Lemma (c) and Lemma 1.
\end{proof}

%section 5
\section{An embedding theorem}

Let  $ P_k(\alpha(0))$ denote the set of points of $F$ of period $k$ with  multiplier $\alpha(0)$,  let  $ P_k(\alpha(1))$ denote the set of points of $F$ of period $k$ with  multiplier $\alpha(1)$, and let $P_k(\bold 1)$ denote  the set of points $F$  periodic points of period $k$, that are neutral.
Denote by $\zeta_\bold 1$ the zeta function of the neutral periodic points of $F$, by $\zeta_{\alpha(0)}$ the zeta function of the  periodic points of $F$ with multiplier $\alpha(0)$,
and   by $\zeta_{\alpha(1)}$ the zeta function of the  periodic points of $F$ with multiplier $\alpha(1)$. 

\begin{embedding lemma}
\begin{align*}
& 
 \liminf_{k \to \infty}
\tfrac{1}{k}
 \log \card \  ( P_k(\bold 1)\cup  P_k(\alpha(0)) = 
\\
& 
 \liminf_{k \to \infty}
\tfrac{1}{k} \log \card \  ( P_k(\bold 1)\cup  P_k(\alpha(1)) = 
\tfrac{3}{2}\log 3 - \log 2.
\end{align*}
\end{embedding lemma}

\begin{proof} Set
$$
\xi(z) = \frac{2}{\sqrt{3}} 
\sin(\frac{1}{3}\arcsin\frac{3\sqrt{3}}{2} z), 
\qquad 0 \leq z \le
\frac{2}{3 \sqrt{3}}.
$$
By \cite [(4.8)] {KM} we have the generating functions
\begin{align*}
 g_ {C^\circ(1)^\star}(z)=  \frac{\xi(z)}{z},   \tag {5.1}
\end{align*}
and
\begin{align*}
g_{C^\star}(z)  =  \frac{\xi(z)^2}{z^2},  \tag {5.2}
\end{align*}
and it follows that
\begin{align*}
\zeta_\bold 1(z) =  g_{\mathcal C^\circ(1)^\star}(z)g_{\mathcal C^\star}(z) =  \frac{\xi(z)^3}{z^3}.  \tag {5.3}
\end{align*}
Using the circular code $\mathcal C^\star\beta^-(0)$ one finds from (5.2) that
\begin{align*}
\zeta_{\alpha(0)}(z) =( \frac{z}{z - \xi(z)^2})^2,\tag {5.4}
\end{align*}
and using the circular code  $\beta^-(1) \mathcal C^\star \beta^- \mathcal C^\circ(1)^\star$  one finds from(5.1) and (5.2)  that
\begin{align*}
\zeta_{\alpha(1)}(z) =( \frac{z}{z - \xi(z)^3})^2.\tag {5.5}
\end{align*}
(E.g. see \cite [Section 5]{P} or \cite [Section 2]{KM}).
The lemma follows from (5.4), (5.5) and (5.3).
\end{proof}

Let  $ P^+_k$ denote the set of points of $F$ of period $k$  with positive multiplier,
 $k\in \Bbb N$. Let  $\mathcal O^+_k$ denote the set of orbits of length $k$ of $F$ with positive multiplier, let  
$\mathcal O_k(\alpha(0))$ denote the set of  orbits of length $k$ of $F$ with  multiplier
$\alpha(0)$,  let $\mathcal O_k(\alpha(1))$ denote the set of  orbits of length $k$ of $F$ with  multiplier
$\alpha(1)$, and
let $\mathcal O_k(\bold 1)$ denote  the set of  orbits of length $k$ of $F$, that are neutral, $k\in \Bbb N$. 

\begin{embedding}
Let $Y$ be an irreducible subshift of finite type. Let $\mathcal O_k(Y)$ be its set of periodic orbits of length $k \in \Bbb N$, and let $h_Y$ be its entropy.
An embedding of Y into the Fibonacci-Dyck shift exists
 if and only if at least one of the following conditions is satisfied:
\bigskip

\noindent
(a)
$$
\card  ( \mathcal O_k(Y)) \leq \card (\mathcal O_k(\bold 1) \cup  \mathcal O_k(\alpha(0))  ), \quad k \in \Bbb N, 
$$
and 
$$
 h(Y) <\tfrac{3}{2}\log 3 - \log 2.
 $$
 
 \noindent
(b)
$$
\card  ( \mathcal O_k(Y) )\leq \card (\mathcal O_k(\bold 1) \cup  \mathcal O_k(\alpha(1))  ), \quad k \in \Bbb N, 
$$
and 
$$
 h(Y) <\tfrac{3}{2}\log 3 - \log 2.
 $$
 
 \noindent
 (c)
 $$
\card  ( \mathcal O_k(Y)) \leq \card (\mathcal O_k(\bold 1) \cup  \mathcal O_k^+  ),  \quad k \in \Bbb N, 
$$
 and 
$$
 h(Y) <3\log 2 - \log 3.
 $$
\end{embedding}

\begin{proof} 
The theorem results from an application of  Theorem 5.8 of \cite {HIK}, that uses Lemma (a) and Lemma 1, Lemma 2 and  Lemma 3, and that takes into account that only the exceptional multipliers, in this case the multipliers $\alpha(0)$ and $\alpha(1)$, contribute to the possibility of an embedding beyond the case of negative (or positive) multipliers.
In Theorem 5.8 of \cite {HIK} the entropy condition in (a) reads
$$
 h(Y) < \liminf _{k \to \infty}\tfrac{1}{k} \log \card  ( P_k(\bold 1)\cup P_k(\alpha(0))),
 $$
the entropy condition in (b) reads
$$
 h(Y) < \liminf _{k \to \infty}\tfrac{1}{k} \log \card  ( P_k(\bold 1)\cup P_k(\alpha(1))),
 $$
and  the entropy condition in (c) reads
$$
 h(Y) < \liminf _{k \to \infty}\tfrac{1}{k} \log \card  ( P_k(\bold 1)\cup P^+_k).
 $$
For (a) and (b)
apply Lemma 5.1.
For (c) note  that by  Lemma (a) the right-hand side of the inequality is equal to the topological entropy of $F$, which is known to be $3\log 2 - \log 3$ \cite [Section 4]{KM}.
\end{proof}

By Lemma (a)
\begin{align*}
&\card ( \mathcal O_k^+) = \tfrac {1}{2}\card ( \mathcal O_k \setminus 
 \mathcal O_{k}(\bold 1)), \tag {5.6}
\\
&\card ( \mathcal O_k(\alpha(0)) = 2 \card ( \mathcal O_k(\alpha^-(0)), 
\\
&\card ( \mathcal O_k(\alpha(1)) = 2 \card ( \mathcal O_k(\alpha^-(1)), \qquad k \in \Bbb N.
\end{align*}
Denote the set of points of $F$ of period $k$ by $P_k,k\in \Bbb N$, and denote by $\zeta$ the zeta function of $F$. 
The sequence 
$(\card( \mathcal O_k))_{k \in \Bbb N}$ can be obtained by M\"obius inversion
from the sequence $(\card (P_k)_{k \in \Bbb N}$ that enters into the zeta function
$
\zeta(z) = e^{\sum_{n \in {\Bbb N}} \frac{\card (P_n)z^n}{n}}
$
 of $F$.
The same  applies to the sequences 
$$
\card (\mathcal O_k(\bold 1) )_{k \in \Bbb N},  \
\card (   P_k(\alpha(0))  )_{k \in \Bbb N}, \
\card (  \mathcal O_k(\alpha(1))  )_{k \in \Bbb N}.
 $$
Therefore, by (5.6), the information that is relevant for each of the conditions (a), (b) and (c) of Theorem (5.1) is (in principle) contained in the zeta functions $ \zeta, \zeta_{ \bold 1}, $ and $\zeta_{\alpha (0)},\zeta_{\alpha (1)}.$
$ \zeta$ is  also  known. From \cite [(4.12)]{KM}
$$
\zeta(z) = \frac {\xi(z)}{z(2  \xi(z)^2 +\xi(z) -1   )^2}.
$$
We note that the zeta function  $\zeta_{ +}(z)$ of the  periodic point with positive multiplier is also known: By Lemma (a)
$$
\zeta_{ +}(z)= \sqrt{\zeta_\bold 1(z)^{-1}\zeta(z)} = \frac{\xi(z)^2}{z^2
(2  \xi(z)^2 +\xi(z) -1   )}.
$$

%section6
\section{The multiplier $\alpha(0)$}
\begin{proposition0} 
The multiplier $\alpha(0)$  is exceptional only at periods one, three and  five.
\end{proposition0}

\begin{proof}  
By Lemma (b) the multiplier $\alpha(0)$ is  
is not exceptional at periods two, four and six. Let
\begin{align*}
n \geq 7.  \tag a
\end{align*}

We construct a shift commuting injection 
$$
\eta_n:P^\circ_n(\alpha_-(0))\rightarrow \bigcup_{\widetilde\mu\in\mathcal M \setminus\{ \alpha(0) \} }
P^\circ_n(\widetilde\mu_-).
$$

Set 
$$
I(i) = \max \{i^{(0)}
 \in   \mathcal I^{(0)}: i^{(0)} < i  \},
 \quad i \in \mathcal I^{(0)}(p), \qquad p \in P^\circ_n(\alpha^-(0)).
 $$
 
 \medskip
Let $P^{(1)}_n$ be the set of points 
$
p\in P^\circ_n(\alpha^-(0))
$
such that the set $\mathcal I^{\langle 1 \rangle}$  of indices $i^0 \in \mathcal I^{(0)}$, such that $p_{(I(i), i  ) }\in \mathcal Q_1$, is not empty.

For $p \in P^{(1)}_n$ we denote by $\mathcal I^{\langle 1 \rangle}_\circ(p)$ the set of indices $i^{(1)}_\circ \in \mathcal I^{\langle 1 \rangle}(p)$, such that the word 
$p_{(I(i^{(1)}_\circ), i^{(1)}_\circ]}$ is lexicographically the smallest  among the words
$$
p_{(I(i^{(1)}), i^{(1)}]},\qquad i^{(1)} \in \mathcal I^{\langle 1 \rangle}(p).
$$

With the word $f(p)\in \mathcal Q_1$, that is given by 
writing  
$$
p_{(I(i^{(1)}_\circ), i^{(1)}_\circ]}= f(p)\beta_-(0),\qquad i_\circ^{(1)} \in\mathcal I^{\langle 1 \rangle}_\circ(p) ,
$$
the shift commuting map $\eta_n$ is to map a point $p \in P^{(1)}_n$ to the point $p \in P^\circ_n(F)$
that is obtained by replacing in the point $p$  the words
$
f(p)\beta_-(0)$, that appear at the indices in $\mathcal I^{\langle 1 \rangle}_\circ(p)$
by $\Delta_1(f)\beta_-(0)$.

For $q \in \eta_n(P^{(1)}_n)$ there is a unique word $b(q) \in\mathcal B(1,1)$, that appears openly in $q$, and
a point $p\in P^{(1)}_n$ can be reconstructed from from its image $q$ under $\eta_n$ by replacing in $q$ the word  $b(q)$,  when it appears openly  in $q$, by the word  
$\Phi_1^{-1}(b(q))$. 

We note that
 \begin{align*}
( \nu_0(\lambda(\eta(p)_{[0, n)})), \nu_1(\lambda(\eta(p)_{[0, n)}))) = 
(\kappa_p   ,  2 ),\quad p \in P^{(1)}_n.\tag {1}
\end{align*}

 \medskip

Let $P^{(\beta )}_n$ be the set of points 
$$
p\in P^\circ_n(\alpha^-(0))\setminus P^{(1)}_n)
$$
such that the set $\mathcal I^{\langle \beta \rangle}(p)$ 
of indices $i^{(\beta)} \in \mathcal I^{(0)}(p)$ at which there appears a word in
$\beta^-\beta^+(\mathcal C(0)^\ast \setminus \{ \epsilon \})\beta^-(0) $ 
is not empty.
 
For $p \in P^{(\beta)}_n$ we denote by $\mathcal I^{\langle \beta \rangle}_\circ(p)$ the set of indices $i^{(\beta)}_\circ \in \mathcal I^{\langle \beta \rangle}(p)$ such that the word 
$p_{(i^{(\beta)}_\circ-n, i^{(\beta)}_\circ]}$ is lexicographically the smallest  among the words
$$
p_{(i^{(\beta)}-n, i^{(\beta)}]}, \quad  i^{(\beta)}\in \mathcal I^{\langle \beta \rangle}(p).
$$

With the word $f(p)\in \mathcal \mathcal C(0)^\ast$, that is given by 
writing  
$$
p_{(I(i^{(\beta)}_\circ), i^{(\beta)}_\circ]}= f(p),\qquad i_\circ^{(\beta)} \in\mathcal I^{\langle \beta \rangle}_\circ(p) ,
$$
the shift commuting map $\eta_n$ is to map a point $p \in P^{(\beta)}_n$ to the point $p \in P^\circ_n(F)$
that is obtained by replacing in the point $p$  the word $\beta^-\beta^+f(p)
 $, that appears in $p$ at the indices in
$
\mathcal I^{\langle \beta \rangle}_\circ(p)-1,
$
by the word $\beta^-\Psi(f(p))\beta^-(1)$.

For a  point $q\in \eta_n(P^{(\beta)}_n )$ there is a unique word $b(q) \in  \mathcal B(1)$,  that appears openly  in $q$, and
a point $p\in P^{(\beta)}_n$ can be reconstructed from  its image $q$ under $\eta_n$ by replacing in $q$ the word $b(q)$, when it appears openly in $q$,
by the word  
$\beta^-\beta^+\Psi^{-1}(b(q))$. 

We note that
\begin{align*}
( \nu_0(\lambda(\eta(p)_{[0, n)})), \nu_1(\lambda(\eta(p)_{[0, n)}))) = 
(\kappa_p    ,  2 ),\quad  p \in P^{(\beta)}_n.\tag {$\beta$}
\end{align*}

 \medskip

Let $P^{(\beta, 0)}_n$ be the set of points 
$$
p\in P^\circ_n(\alpha^-(0))\setminus (P^{(1)}_n \cup P^{(\beta)}_n)
$$
such that the set $\mathcal I^{\langle \beta, 0 \rangle}(p)$ 
of indices $i^{(\beta, 0)} \in \mathcal I^{(0)}(p)$ at which there appears openly  the word 
$\beta^-\beta^+\beta^-(0) $ 
is not empty.

For $p \in P^{(\beta, 0)}_n$ we denote by $\mathcal I^{\langle \beta, 0 \rangle}_\circ(p)$ the set of indices $i^{(\beta, 0)}_\circ \in \mathcal I^{\langle \beta, 0 \rangle}(p)$
such that the word 
$p_{(i^{(\beta, 0)}_\circ-n, (i^{(\beta, 0)}_\circ]}$ is lexicographically the smallest  among the words
$$
p_{(i^{(\beta, 0)}-n, i^{(\beta, 0)}]},\qquad i^{(\beta, 0)} \in \mathcal I^{\langle \beta, 0 \rangle}(p).
$$

The shift commuting map $\eta_n$ is to map a point $p \in P^{(\beta, 0)}_n$ to the point $p \in P^\circ_n(F)$
that is obtained by replacing in the point $p$  the word $\beta^-\beta^+$, that appears in $p$ at the indices in
$
\mathcal I^{\langle\beta, 0 \rangle}_\circ(p)-1,
$
by the word $\beta^-\beta^-(1)$.

In a point of $\eta_n(P^{(\beta, 0)}_n)$ the word $\beta^-\beta^-(1)$ appears openly,
and a  point $p\in P^{(\beta, 0)}_n$ can be reconstructed from  its image $q$ under $\eta_n$ by replacing in $q$ the word 
$\beta^-\beta^-(1)$,  when it  appears openly  in $q$, by the word  
$\beta^-\beta^+$. 

We note that
\begin{align*}
( \nu_0(\lambda(\eta(p)_{[0, n)})), \nu_1(\lambda(\eta(p)_{[0, n)}))) = 
(\kappa_p    ,  2 ),\quad p \in P^{(\beta,0)}_n.\tag {$\beta$.0}
\end{align*}

\medskip

Let $P^{(0,2)}_n$ be the set of points 
$$
p\in P^\circ_n(\alpha^-(0))\setminus
 (P^{(1)}\cup P^{(\beta)}_n\cup P^{(\beta, 0)}_n),
$$
such that 
$$
\nu_0(p) \geq 2.
$$

For $p \in P^{(0,2)}_n$ we denote by $\mathcal I^{\langle 0,2 \rangle}_\circ(p)$ the set of indices $i \in \mathcal I^{(0)}(p)$, such that
$$
i - I(i) > 1.
$$

For $p \in P^{( 0,2)}_n$ we denote by $\mathcal I^{\langle  0,2 \rangle}_\circ(p)$ the set of indices 
$i^{(0,2)}_\circ \in \mathcal I^{\langle 0,2 \rangle}(p)$ such that the word 
$p_{(i^{(0,2)}_\circ-n, i^{(0,2)}_\circ]}$ is lexicographically the smallest  among the words
$$
p_{(i^{(0,2)}-n, i^{(0,2)}]},\qquad  i^{(0,2)} \in \mathcal I^{\langle 0,2 \rangle}(p).
$$

With the word $f(p) \in \mathcal C(0)^\ast$, that is given by writing
$$
p_{[I(i^{\langle  0,2 \rangle}_\circ),i^{\langle  0,2 \rangle}_\circ ]} = \beta^-(0) f(p)\beta^-(0) , \qquad i^{(0,2)}   \in \mathcal I^{\langle  0,2 \rangle}_\circ(p),
$$
the shift commuting map $\eta_n$ is to map a point $p \in P^{(0,2)}_n$ to the point $p \in P^\circ_n(F)$
that is obtained by replacing in the point $p$  the words $\beta^-(0) f(p)\beta^-(0)$, that appear in $p$ at the indices in $ \mathcal I^{\langle  0,2 \rangle}_\circ(p)$ by the word 
$ \beta^-\Psi^{-1}(f(p)) \beta^-(1)$.

For a  point $q\in \eta_n(P^{( 0,2)}_n )$ there is a unique word
 $b(q) \in  \mathcal B(1)$,  that appears openly  in $q$, and
a point $p\in P^{( 0,2)}_n$ can be reconstructed from  its image $q$ under $\eta_n$ by replacing in $q$ the word $b(q)$, when it appears openly in $q$,
by the word  
$\beta^-\Psi^{-1}(b(q))\beta^-$. 

We note that
 \begin{align*}
( \nu_0(\lambda(\eta(p)_{[0, n)})), \nu_1(\lambda(\eta(p)_{[0, n)}))) = 
(\kappa_p - 2   ,  1 ), \quad p \in P^{(0,2)}_n.\tag {0.2}
\end{align*}
 
\medskip

Let $P^{(0,1,l)}_n$ be the set of points 
$$
p\in P^\circ_n(\alpha^-(0))\setminus  
(P^{(1)}\cup P^{(\beta)}_n\cup P^{(\beta, 0)}_n\cup P^{(0,2)}_n),
$$
such that
$$
p_{(i-n, i]} \in \beta^-(0)\beta^+(0)\mathcal C(0)^\ast \beta^-(0), \qquad i \in\mathcal I^{(0)}(p).
$$

With the word $f(p) \in \beta^-(0)\beta^+(0)\mathcal C(0)^\ast \beta^-(0)$, that is given by writing
$$
p_{(i-n, i]}= \beta^-(0)\beta^+(0)f(p) \beta^-(0), \qquad i \in \mathcal I^{(0)}(p),
$$
the shift commuting map $\eta_n$ is to map a point $p \in P^{(0,1,l)}_n$ to the point 
$q \in P^\circ_n(F)$
that is given by
$$
q_{(i-n, i]}=  \beta^-\beta^-(1) f(p)\beta^-(0), \qquad i \in \mathcal I^{(0)}(p).
$$

For a  point $q\in \eta_n(P^{( 0,1,l)}_n )$ one has
$$
q_{(i-n, i]}\in  \beta^-\beta^-(1) \mathcal C(0)^\ast \beta^-(0), \qquad i \in \mathcal I^{(0)}(q),
$$
and with the word $ a(q) \in \mathcal C(0)^\ast$ that is given by writing
$$
 q_{(i-n, i]}=   \beta^-\beta^-(1)a(q) \beta^-(0), \qquad i \in \mathcal I^{(0)}(q),
$$
a point $p\in P^{( 0,1,l)}_n$ can be reconstructed from its image $q$ under $\eta_n$ as the point in $P^\circ_n(F)$  that is given by
 $$
 p_{(i-n, i]}=   \beta^-(0)\beta^+(0)a(q)) \beta^-(0), \qquad i \in \mathcal I^{(0)}(q).
$$

We note that
 \begin{align*}
( \nu_0(\lambda(\eta(p)_{[0, n)})), \nu_1(\lambda(\eta(p)_{[0, n)}))) = 
(1  ,  1 ), \quad p \in P^{(0,1,l)}_n.\tag {0.1.l}
\end{align*}

 \medskip

Let $P^{(0,1,m)}_n$ be the set of points 
$$
p\in P^\circ_n(\alpha^-(0))\setminus (P^{(1)}_n\cup P^{(\beta)}_n
\cup P^{(\beta, 0)}_n\cup P^{(0, 2)}_n),
$$
such that
$$
p_{(i-n, i]} \in  \beta^-(0)(\mathcal C(0)^\ast \setminus\{ \epsilon  \}) \beta^+(0)  
(\mathcal C(0)^\ast \setminus\{ \epsilon  \})\beta^-(0), \qquad i \in \mathcal I^{(0)}(p). 
$$

With the words $f(p) \in \mathcal C(0)^\ast,g(p) \in \mathcal C(0)^\ast$, that are given by writing
$$
p_{(i-n, i]}=
\beta^-(0)f(p) \beta^+(0)  g(p)\beta^-(0),\qquad i \in \mathcal I^{(0)}(p),
$$
the shift commuting map $\eta_n$ is to map a point $p \in P^{(0,1,m)}_n$ to the point $q \in P^\circ_n(F)$
that is given by
$$
q_{(i-n, i]}= \beta^-(1)f(p) \beta^-(0)g(p)\beta^- , \qquad i \in \mathcal I^{(0)}(p).
$$

For a  point $q\in \eta_n(P^{( 0,1.m)}_n )$ one has
$$
q_{(i-n, i]}\in  \mathcal C(0)^\ast \beta^-(0) \mathcal C(0)^\ast\beta^- \beta^-(1), 
\qquad i \in \mathcal I^{(1)}(q),
$$
and with the words $a(q)\in  \mathcal C(0)^\ast, b(q)\in \mathcal C^\ast$, that are given by writing
$$
q_{[i, i + n}=  a(q) \beta^-(0) b(q)\beta^- \beta^-(1). \qquad i \in 
\mathcal I^{(1)}(q),
$$
a point $p\in P^{( 0,1,0)}_n$ can be reconstructed from its image $q$ under $\eta_n$  
as the point that is given by
$$
p_{(i-n, i]}=a(q) \beta^+(0)b(q)\beta^-(0) \beta^-(0), \qquad i \in 
\mathcal I^{(1)}(q).
$$

We note that
 \begin{align*}
( \nu_0(\lambda(\eta(p)_{[0, n)})), \nu_1(\lambda(\eta(p)_{[0, n)}))) = 
(1  ,  1 ), \quad  p \in P^{(0,1,0)}_n.\tag {0.1.m}
\end{align*}

 \medskip
 
Let 
$
P^{(0,1,r)}_n$ be the set of points
$$
p\in P^\circ_n(\alpha^-(0))\setminus (P^{(1)}_n\cup P^{(\beta)}_n
\cup P^{(\beta, 0)}_n\cup P^{(0, 2)}_n),
$$
such that
$$
p_{(i-n, i]} \in  \beta^-(0) \beta^-(0)(\mathcal C(0)^\ast \setminus\{ \epsilon  \}) \beta^+(0)  
\mathcal C(0)^\ast \beta^+(0)\beta^-(0), \qquad i \in \mathcal I^{(0)}(p). 
$$

With the words $f(p) \in \mathcal C(0)^\ast,g(p) \in \mathcal C(0)^\ast$, that are given by writing
$$
p_{(i-n, i]}=
\beta^-(0)f(p) \beta^+(0)  g(p)\beta^-(0),\qquad i \in \mathcal I^{(0)}(p),
$$
the shift commuting map $\eta_n$ is to map a point $p \in P^{(0,1,r)}_n$ to the point $q \in P^\circ_n(F)$
that is given by
$$
q_{(i-n, i]}=\beta^-\beta^-(1)f(p)\beta^-(0)g(p)\beta^-(0)\beta^-(0) , \qquad i \in \mathcal I^{(0)}(p).
$$

For a  point $q\in \eta_n(P^{( 0,1,r)}_n )$ one has
$$
q_{(i-n, i]}\in  \mathcal C(0)^\ast \beta^-(0) \mathcal C(0)^\ast
 \beta^-(0) \beta^-(0)\beta^- \beta^-(1), 
\qquad i \in \mathcal I^{(1)}(q),
$$
and with the words $a(q)\in  \mathcal C(0)^\ast, b(q)\in \mathcal C(0)^\ast$, that are given by writing
$$
q_{[i, i + n}=  a(q) \beta^-(0) b(q)\beta^-(0)\beta^-(0)\beta^- \beta^-(1). \qquad i \in 
\mathcal I^{(1)}(q),
$$
a point $p\in P^{( 0,1,r)}_n$ can be reconstructed from its image $q$ under $\eta_n$  
as the point that is given by
$$
p_{(i-n, i]}=a(q) \beta^-(0) b(q) \beta^-(0)\beta^- \beta^-(1), 
\qquad i \in \mathcal I^{(1)}(q).
$$

We note that
 \begin{align*}
( \nu_0(\lambda(\eta(p)_{[0, n)})), \nu_1(\lambda(\eta(p)_{[0, n)}))) = 
(3  ,  1 ), \quad  p \in P^{(0,1,r)}_n.\tag {0.1.r}
\end{align*}

Let $P^{(0,1,r,\beta)}_n$ be the set of points 
$$
p\in P^\circ_n(\alpha^-(0))\setminus  
(P^{(1)}\cup P^{(\beta)}_n\cup P^{(\beta, 0)}_n\cup P^{(0,2)}_n),
$$
such that
$$
p_{(i-n, i]} \in \beta^-(0)\beta^-(0)\beta^+(0)
(\mathcal C(0)^\ast  \setminus\mathcal C(0)^\ast)
\beta^+(0)\beta^-(0), \quad i \in\mathcal I^{(0)}(p).
$$

With the word $f(p) \in  \mathcal C(0)^\ast $, that is given by writing
$$
p_{(i-n, i]}=  \beta^-(0)\beta^-(0)\beta^+(0)
f(p)
\beta^+(0)\beta^-(0), \quad i \in \mathcal I^{(0)}(p),
$$
the shift commuting map $\eta_n$ is to map a point $p \in P^{(0,1,r,\beta))}_n$ to the point 
$q \in P^\circ_n(F)$
that is given by
$$
q_{(i-n, i]}=  \beta^-(0)\beta^-\beta^-(1) f(p)\beta^-(0)\beta^-(0), \qquad i \in \mathcal I^{(0)}(p).
$$

For a  point $q\in \eta_n(P^{( 0,1,r,\beta)}_n )$ one has
$$
q_{(i-n, i]}\in \mathcal C(0)^\ast \beta^-(0) \beta^-(0) \beta^-(0)\beta^-\beta^-(1),  \qquad i \in \mathcal I^{(1)}(q),
$$
and with the word $a(q) \in \mathcal C(0)^\ast$ that is given by writing
$$
 q_{(i-n, i]}=   a(q) \beta^-(0)\beta^-\beta^-(1), \qquad i \in \mathcal I^{(0)}(q),
$$
a point $p\in P^{( 0,1,r,\beta)}_n$ can be reconstructed from its image $q$ under $\eta_n$ as the point in $P^\circ_n(F)$  that is given by
 $$
 p_{(i-n, i]}= a(q)  \beta^+(0) \beta^-(0) \beta^-(0)\beta^+(0), \qquad i \in \mathcal I^{(0)}(q).
$$

We note that
 \begin{align*}
( \nu_0(\lambda(\eta(p)_{[0, n)})), \nu_1(\lambda(\eta(p)_{[0, n)}))) = 
(3  ,  1 ), \quad p \in P^{(0,1,r,\beta)}_n.\tag {0.1.r.$\beta$}
\end{align*}

Let $P^{(0,1,r,1))}_n$ be the set of points 
$$
p\in P^\circ_n(\alpha^-(0))\setminus  
(P^{(1)}\cup P^{(\beta)}_n\cup P^{(\beta, 0)}_n\cup P^{(0,2)}_n),
$$
such that
$$
p_{(i-n, i]} \in\beta^-(0) \beta^-(0)\beta^+(0)\mathcal C(0)^\ast\beta^+(0) \beta^-(0), \qquad i \in\mathcal I^{(0)}(p).
$$

With the word $f(p) \in \beta^-(0)\beta^+(0)\mathcal C(0)^\ast \beta^-(0)$, that is given by writing
$$
p_{(i-n, i]}= \beta^-(0) \beta^-(0)\beta^+(0)f(p)\beta^+(0) \beta^-(0), \quad i \in \mathcal I^{(0)}(p),
$$
the shift commuting map $\eta_n$ is to map a point $p \in P^{(0,1,r,1))}_n$ to the point 
$q \in P^\circ_n(F)$
that is given by
$$
q_{(i-n, i]}=  \beta^-(1)\beta^- \Psi^{-1}(f(p))\beta^-(1) \beta^-(0)\beta^-\, \qquad i \in \mathcal I^{(0)}(p).
$$

For a  point $q\in \eta_n(P^{( 0,1,r,1)}_n )$ one has
$$
q_{(i-n, i]}\in  \beta^-\beta^-(1) \beta^-\mathcal C(0)^\ast \beta^-(1)\beta^-(0), \qquad i \in \mathcal I^{(0)}(q),
$$
and with the word $ a(q) \in \mathcal C(0)^\ast$ that is given by writing
$$
 q_{(i-n, i]}=   \beta^-\beta^-(1)a(q) \beta^-(0), \qquad i \in \mathcal I^{(0)}(q),
$$
a point $p\in P^{( 0,1,r,\beta))}_n$ can be reconstructed from its image $q$ under $\eta_n$ as the point in $P^\circ_n(F)$  that is given by
 $$
 p_{(i-n, i]}=   \beta^-\beta^-(1) \beta^-\Psi(a(q)) \beta^-(1)\beta^-(0), \qquad i \in \mathcal I^{(0)}(q).
$$

We note that
 \begin{align*}
( \nu_0(\lambda(\eta(p)_{[0, n)})), \nu_1(\lambda(\eta(p)_{[0, n)}))) = 
(1  ,  2 ), \quad p \in P^{(0,1,r,0)}_n.\tag {0.1.r.0}
\end{align*}

An inspection of the definition of $P^{(0, 1, r,1)}_n$ shows that we have produced a partition 
 \begin{multline*}
P^\circ_n(\alpha_-(0)) =P^{(1)}_n\cup P^{(\beta,\beta)}_n\cup P^{(\beta)}_n\cup P^{(\beta,0)}_n\cup P^{(0, 2)}\cup\tag {P.0}
\\
P^{(0, 1,l)}_n\cup P^{(0, 1, m)}_n\cup P^{(0, 1, r)}_n\cup P^{(0,1,r,\beta)}_n\cup P^{(0,1,r,0)}_n.
\end{multline*}

The points of $\eta_n( P^{(1)}_n)$ are the only ones in $\eta_n( P^\circ_n(\alpha_-(0))$, in which there appears openly a word in $\mathcal B(1.1)$.
In the points of $\eta_n(P^{(\beta)}_n   \cup P^{(0,2)}_n)$ the word 
$\beta^-\beta^-(1)  $ does not appear openly, whereas in the words of 
$\eta_n(P^{(0, 1,l)}_n\cup P^{(0, 1, m)}_n\cup P^{(0, 1, r)}_n\cup P^{(0,1,r,\beta)}_n\cup P^{(0,1,r,0)}_n)$ this word does appear openly.
In the points of $\eta_n( P^{(\beta, 0)}_n)$ there appears the word $ \beta^-\beta^-(1) \beta^-(0)$ openly, whereas in the In the points of $\eta_n(P^{(0, 1,l)}_n\cup P^{(0, 1, m)}_n\cup P^{(0, 1, r)}_n)$ this word does not appear openly. However in the points of $\eta_n(P^{(0, 1,l)}_n\cup P^{(0, 1, m)}_n\cup P^{(0, 1, r)}_n)$ the word  $\beta^-\beta^-(1)  $  does appear openly.
The points of $\eta_n( P^{(0, 1, r,\beta)}_n)$ are the only ones in $\eta_n( P^\circ_n(\alpha_-(0))$, in which there appears openly a word in  $(\mathcal C^\ast  \setminus \mathcal C(0)^\ast)  \beta^-(0)$,
and   the points of $\eta_n( P^{(0, 1, r,1)}_n)$ are the only ones in $\eta_n( P^\circ_n(\alpha_-(0))$, in which there appears openly a word in $\beta^- \beta^-(1)\mathcal C(0)^\ast\mathcal B(1)\mathcal C(0)^\ast $.

From   these   observations and   from  (1), ($\beta$), ($\beta$.0), (0.2), and (0.1.l), (0.1.m), (0.1.r), (0.1.r.$\beta$), (0.1.r.0)
 it follows,  that the 
images under $\eta_n$ of the sets in the partition (P.0) are disjoint. Also 
$\eta_n(   P^\circ_n(\alpha_-(0) ))\cap P^\circ_n(\alpha_-(0) )  =\emptyset$.
We have shown that
$$
\card \  \mathcal O_n(\alpha^-(0)) \leq \card ( \bigcup_{\widetilde{\mu}\in \mathcal M \setminus \{ \alpha^- (0)\}}  \mathcal O_n(\widetilde{\mu}^-)).
$$
Apply Lemma (c).
\end{proof}

\section{The multiplier $\alpha(1)$}

\begin{proposition1} 
 The  multiplier $\alpha(1)$ is  exceptional only at period two.
\end{proposition1}

\begin{proof}  
We construct shift commuting injections
$$
\eta_n  :P_n^\circ(\alpha_-(1))
\rightarrow
 \bigcup_{\tilde\mu\in \mathcal M\setminus \{\alpha_-(1)\}}
P^\circ_n(\tilde\mu^-),    \quad n > 2.
$$

Let $n> 2$. Denote by $P^{(1)}_n$  the set of $p\in P^\circ_n(\alpha^-(1))$ such that the word $\beta^- \beta^-(1)$ appears openly in $p$. The shift commuting map $\eta_n$ is to map a point $p\in P^{(1)}_n$ to the point $q\in P_n(F)$ that is obtained by replacing in $p$ the word $\beta^- \beta^-(1)$,   when it appears openly in $p$,
by the word $ \beta^-(0)\beta^-(0)  $. 
A point $p\in P^{(1)}_n$ can be reconstructed from its image  $q$ under $\eta_n$ by replacing in $q$  every word 
$
  \beta^-(0)^{2K} $
that  appears  in $q$ openly, 
and that is neither preceded nor followed in $q$ by an open appearance of the symbol $\beta^-(0)$, by the word 
 $( \beta^- \beta^-(1) )^K $, $K \in \Bbb N$. 

Denote by $P^{(2)}_n$  the set of 
$$
p\in P^\circ_n(\alpha^-(1))\setminus P^{(1)}_n
$$
such that  words in
$ \beta^-\mathcal C^\circ(1)^\ast
\beta^-(1)$ appear openly in $p$ at least twice during a period.
Denote  for $p \in P^{(2)}_n $ by $\mathcal J(p)$ the set of indices $j \in \mathcal I^{(1)}(p)$ 
such that the word $p_{(n-j,n]}$ is lexicographically the smallest one among the words
$
p_{(n-i,n]}, i \in  \mathcal I^{(1)}(p).
$
With the word $f^\circ(p) \in \mathcal C^\circ(1)^\ast$, that is given by writing the word in 
$\beta^-(\mathcal C^\circ(1)^\ast\setminus\{ \epsilon \})\beta^-(1)$ that appears  in $p$ at an index in 
$  \mathcal J(p)$ as
\begin{align*}
\beta^- f^\circ(p) \beta^-(1), \tag {a}
\end{align*}
the shift commuting map $\eta_n$ is to map a point $p \in P^{(2)}_n$ to the point $q \in P_n(F)$
that is obtained by replacing in the point $p$ each of the words (a)
that appear in $p$ at an index in $ \mathcal J(p)$
by the word
$$
\beta^-(0)\Psi(f^\circ(p)(p)) \beta^-(0).
$$

With the word $f(q)\in\mathcal C(0)^\ast $ such that the word $ \beta^-(0)  f(q) \beta^-(0)$ appears openly in $q$,
the point $p$ can be reconstructed from its image $q$ under $\eta_n$ by replacing in $q$ the word
 $$
 \beta^-(0)  f(q) \beta^-(0)
 $$ 
when it appears  openly in $q$,  by the word 
$$
\beta^-\Psi^{-1}( f(q) ) \beta^-(1).
$$ 

Denote by 
$
P^{(3)}_n$  the set of 
$$
p\in P^\circ_n(\alpha^-(1))\setminus (P^{(1)}_n\cup P^{(2)}_n )
$$ 
such that
$$
p_{(i-n, i ]}\in(\mathcal C^\ast\setminus \{\epsilon\}) \beta^- \mathcal C^\circ(1)^\ast \beta^-(1).
$$
With the words
 $$
f(p) \in   \mathcal C^\ast ,     \quad f^\circ(p) \in   \mathcal C^\circ(1)^\ast   ,
 $$
 that are given by writing 
\begin{align*}
p_{(i-n, i ]}=f(p) \beta^-f^\circ(p)\beta^-(1),
\end{align*}
and denoting  by $\Psi^\prime(c)$
the word that is obtained by removing from the word $\Psi(f^\circ(p))$ its last symbol,
the shift commuting map $\psi_n$ is to map a point $p \in  P^{(3)}_n$  to the point $q\in P^\circ_n(F)$ that is given by
$$
q_{(i-n, i ]}  =f(p) \beta^-(0)\Psi^\prime(f^\circ(p))\beta^- \beta^-(1), \quad i \in \mathcal I^{(1)}(q).
$$

With the  words 
$$
f(q) \in \mathcal C^\ast\setminus \{ \epsilon \}, \quad
f^\prime(q)\in\mathcal C(0)^\ast\beta^-(0)\mathcal C(0)^\ast,
$$
that are given by writing
$$
q_{(i-n, i ]} =f(q)\beta^-(0)f^\prime(q) \beta^-\beta^-(1), \quad i \in \mathcal I^{(1)}(q),
$$
the point $p$ can be reconstructed from its image $q$ under $\eta_n$ as the point that is given by
$$
p_{(i-n, i ]}  =f(q)\beta^-\Psi^{-1}(f^\prime(q)\beta^-(0)) \beta^-(1), \quad i \in \mathcal I^{(1)}(q).
$$

Set 
$$
P^{(4)}_n= P^\circ_n(\alpha^-(1))\setminus (P^{(1)}_n\cup P^{(2)}_n\cup P^{(3)}_n).
$$ 
One has 
 $$
 p_{(i-n, i]} \in \beta^-(1)\beta^-\mathcal C^\circ(1)^\ast, \quad i \in \mathcal I^{(1)}(p),
 \qquad p \in P^{(4)}_n.
 $$
 With the words
 $$
 c(p) \in \mathcal C^\circ(1), \quad f^\circ(p) \in \mathcal C^\circ(1)^\ast,
 $$
 that are given by writing
  $$
 p_{(i-n, i]}=  \beta^-(1)\beta^-c(p)  f^\circ(p),\quad i \in \mathcal I^{(1)}(p),
 $$
the shift commuting map $\eta_n$ is to map a point $p \in P^{(4)}_n$ to the point $q \in P_n(F)$
that is given by
$$
q_{(i-n, i]} = \beta^-(1) \Phi_0(  \Psi ( c(p)))\Psi(f^\circ(p))\beta^-,\quad i \in \mathcal I^{(1)}(p).
$$

With the  words
$$
b(q) \in \mathcal B(0,0), \quad 
g(q)\in   \mathcal C (0)^\ast  ,
$$
that are given by writing
$$
q_{(i-n, i]} =\beta^-(1)b(q)g(q)\beta^-, \quad i \in \mathcal I^{(1)}(q),
$$
the point $p$ can be reconstructed from its image $q$ under $\eta_n$  as the point that is given by
$$
p_{(i-n, i]} =\beta^-(1)\beta^-   \Psi^{-1}    ( \Phi_0^{-1}(b(q)) )   \Psi^{-1}(g(q)),\quad i \in \mathcal I^{(1)}(q).
$$

We have produced a partition
\begin{align*}
P^\circ_n(\alpha_-(1)) = \bigcup_{1\leq l \leq 4}P^{(l) }_n . \tag {P.1}
\end{align*}
In the points in $\eta_n(P^{(1) }_n )$  there appears  openly the word 
$\beta^- (0)\beta^-(0)$, and  the word 
$\beta^- \beta^- (1)$ does not appear openly, and in
the points in $\eta_n(P^{(2) }_n ) $  neither the word $\beta^- (0)\beta^- (0)$ nor the word 
$\beta^- \beta^- (1)$ appears  openly.
In the points in $\eta_n(P^{(3) }_n) $ and $\eta_n(P^{(4) }_n )$ there appears the word 
$\beta^- \beta^- (1)$ openly. Also
\begin{align*}
&(\mathcal I^{(1)} (q) + 1  ) \cap \mathcal I^{(0)} (q) = \emptyset, \quad   
   q \in  \eta_n(P^{(3) }_n) ,
\\
&\mathcal I^{(1)} (q) + 1  \subset \mathcal I^{(0)}, \ \  \quad \quad \quad \quad  q \in \eta_n(P^{(4) }_n)   .
\end{align*}
From these observations  it follows  that the 
images under $\eta_n$ of the sets of the partition (P.1)
are disjoint.

We have shown that
$$
\card (  \mathcal O_n(\alpha^-(0))) \leq \card ( \bigcup_{\widetilde{\mu}\in \mathcal M \setminus \{ \alpha (1)\}}  \mathcal O_n(\widetilde{\mu}_-)).
$$
Apply Lemma (c).
\end{proof}

\medskip 

\par\noindent Toshihiro Hamachi
\par\noindent  Faculty of Mathematics
\par\noindent Kyushu University
\par\noindent 744 Motooka, 
 Nishi-ku,
Fukuoka 819-0395, Japan
\par\noindent hamachi@math.kyushu-u.ac.jp

\bigskip

\par\noindent Wolfgang Krieger,
\par\noindent Institute for Applied Mathematics,
\par\noindent  University of Heidelberg
\par\noindent Im Neuenheimer Feld 294, 
 69120 Heidelberg, Germany
\par\noindent krieger@math.uni-heidelberg.de

 \end{document}